\documentclass[10pt, reqno]{amsart}
\usepackage{color}
\usepackage{amsmath}
\usepackage{amsthm}
\usepackage{amssymb}
\usepackage{amsfonts}
\usepackage{mathrsfs}
\usepackage{amsbsy}
\usepackage{relsize}
\usepackage[colorlinks, linkcolor=red, citecolor=blue, urlcolor=blue, pagebackref, hypertexnames=false]{hyperref}
\usepackage[hyperpageref]{backref}
\usepackage[framed, numbered]{matlab-prettifier}
\usepackage{graphicx, epstopdf}
\usepackage{epsfig}
\usepackage{graphics}
\usepackage{eufrak}

\usepackage[lmargin=3cm,rmargin=3cm,tmargin=3cm,bmargin=3cm]{geometry}
\newcommand{\beq}{\begin{eqnarray}}
\newcommand{\eeq}{\end{eqnarray}}
\newcommand{\bq}{\begin{equation}}
\newcommand{\eq}{\end{equation}}
\newcommand{\beqn}{\begin{eqnarray*}}
\newcommand{\eeqn}{\end{eqnarray*}}

\newcommand{\vertiii}[1]{{\vert\kern-0.25ex\vert\kern-0.25ex\vert #1
    \vert\kern-0.25ex\vert\kern-0.25ex\vert}}

\newcommand{\ignore}[1]{}

\newtheorem{definition}{Definition}[section]
\newtheorem{proposition}{Proposition}[section]
\newtheorem{theorem}{Theorem}[section]

\newtheorem{remark}{Remark}[section]

\newtheorem{lemma}{Lemma}[section]
\newtheorem{corollary}{Corollary}[section]
\newtheorem{notation}{Notation}[section]
\numberwithin{equation}{section}
\title[On the LRD of time-changed mfBm model]{On the Long range Dependence of time-changed mixed fractional Brownian motion model}
\author[S. Alajmi, E. Mliki]{Shaykhah  Alajmi$^{1}$ and Ezzedine Mliki$^{1, 2}$}
\address{$^{1}$Department of Mathematics, College of Science, Imam Abdulrahman Bin Faisal University, P. O. Box 1982, Dammam, Saudi Arabia.
$^{2}$Basic and Applied Scientific Research Center, Imam Abdulrahman Bin Faisal University, P.O. Box 1982, Dammam, 31441, Saudi Arabia. }
\email{\sl 2190500234@iau.edu.sa}
\email{\sl ermliki@iau.edu.sa}

\begin{document}
\begin{abstract} A  time-changed mixed fractional Brownian motion  is an iterated process  constructed as the superposition of mixed fractional Brownian motion and other process.   In this paper we consider mixed fractional Brownian motion of parameters $a, b$ and $H\in (0, \, 1)$ time-changed by two processes,  gamma and tempered stable subordinators. We present their main properties paying main attention to the long range dependence. We deduce that the fractional Brownian motion time-changed by gamma and tempered stable subordinators has long range dependence  property for all $H\in(0,1)$.
\end{abstract}





\maketitle

\section{Introduction}

The fractional  Brownian motion (fBm)  $B^{H}=\{B_{t}^{H},\, t\geq 0 \}$ with parameter $H$,  is a centered Gaussian process with covariance function
\begin{eqnarray}
	\mbox{Cov}(B_{t}^{H},\,B_{s}^{H})= \frac{1}{2} [t^{2H}+s^{2H}- |t-s|^{2H}],\quad s,t\geq 0,
\end{eqnarray}
where $H$ is a real number in $(0, 1),$ called the Hurst index or Hurst exponent.  The case $H=\frac{1}{2}$ corresponds to the Brownian motion (Bm).

  An extension of the  fBm was introduced by Cheridito \cite{Cher}, called the  mixed fractional Brownian motion (mfBm for short) which is a linear combination between a Brownian motion and an independent fractional Brownian motion of Hurst exponent $H$, with stationary increments exhibit a long-range dependent for $H>\frac{1}{2}.$ A mfBm of parameters $a,  b $ and $  H $ is a process  $N^H(a,b)=\{N^{H}_{t}(a,b), \; t\geq 0\}$, defined on the probability space  $(\Omega, \mathcal{F}, P)$ by
	\begin{eqnarray*}
		N_t^H(a,b)=aB_t+bB_t^H,
	\end{eqnarray*}
	where  $B=\left\lbrace B_t, t\geq0\right\rbrace $ is a Brownian motion and  $B^H=\left\lbrace B_t^H, t\geq0\right\rbrace $ is an independent fractional Brownian motion of Hurst exponent $H\in(0,1)$. We refer also to \cite{AlMl, Cher, ElNo, FFDD,MMEM, Thale} for further information on mfBm process.

The  time-changed mixed fractional Brownian motion is defined as 	 \begin{eqnarray*}Y^{H}_{\beta}(a,b)=\{Y^{H}_{\beta_{t}}(a,b), \; t\geq 0\} =\{N^{H}_{\beta_{t}}(a,b), \; t\geq 0\}	\end{eqnarray*},  where the parent process $N^{H}(a,b)$ is a mfBm with parameters $a, b,$ $H\in (0, \, 1)$ and  the subordinator $\beta=\{ \beta_{t},\, t\geq 0\}$ is assumed to be independent of both the  Brownian motion and the fractional Brownian motion. If $H=\frac{1}{2}$, the process $Y^{\frac{1}{2}}_{\beta}(0,1)$ is called subordinated Brownian motion, it was investigated in \cite{Mag, Nan}. Also, the process $Y^{H}_{\beta}(0,1)$ is called subordinated fractional Brownian motion  it was investigated in  \cite{KGW, KWPS}.

Time-changed process is constructed by taking superposition of tow independent stochastic systems. The evolution of time in external process is replaced by a non-decreasing stochastic process, called subordinator.  The resulting time-changed process very often retain important properties of the external process, however certain characteristics might change. This idea of subordination was introduced by Bochner \cite{Bochner} and  was  explored in many papers (e.g. \cite{Al,  HmMl, MeHmMl,  KGW, MejMl, MeMl, Omer, Sch, TWS}).

The  time-changed mixed fractional Brownian motion has been discussed in \cite{GZH} to present a stochastic model of the discounted stock price in some arbitrage-free and complete financial markets. This model is the process
\begin{eqnarray*}
X_{t}^{H}(a,b)=X_{0}^{H}(a,b)\exp\{\mu \beta_{t}+\sigma N^{H}_{\beta_{t}}(a,b)\},
\end{eqnarray*}
where $\mu$ is the rate of the return and $\sigma$ is the volatility and $\beta_{t}$ is the $\alpha$-inverse stable subordinator.

The time-changed processes have found many interesting applications, for example in finance  \cite{AlexY, GZH, HJY, Omer, FSH}, in statistical inference \cite{AlexYES} and in physics  \cite{GSPE}.

Our goal in this parer is to study the main properties of  the time-changed mixed fractional Brownian motion model paying attention to the long range dependence property. In the first case the internal process, which plays role of time, is the tempered stable subordinator while  in the second case the internal process is the  gamma subordinator.


\section{MfBm  time-changed by Tempered Stable Subordinator }
\begin{definition}
	Tempered Stable Subordinator  with index $\alpha\in(0, 1)$ and tempering parameter $\lambda>0$ (TSS)  is the non-decreasing and non-negative L\'{e}vy process  $S^{\lambda,\alpha}=\{S^{\lambda,\alpha}_{t}, \; t\geq 0\} $ with density function:
	\begin{eqnarray*}
		f_{\lambda,\alpha}(x,t)=exp(-\lambda x+\lambda^\alpha t)f_{\alpha}(x,t), \ \ \lambda>0, \ \ \alpha\in(0,1),
	\end{eqnarray*}
	where
	\begin{eqnarray*}
		f_{\alpha}(x,t)=\frac{1}{\pi}\int_0^\infty e^{-xy}e^{-ty^\alpha \cos\alpha\pi}\sin(ty^\alpha \sin\alpha\pi)dy .
	\end{eqnarray*}	
More detail about TSS can be founded in \cite{KGW}.
\end{definition}
\begin{lemma}(see \cite{KGW} for the proof)\label{ll1}\\
For $q>0,$ the asymptotic behavior of q-th order moments of $S^{\lambda,\alpha}_{t}$ satisfies
\begin{eqnarray*}
E(S^{\lambda,\alpha}_{t})^{q}\sim (\alpha \lambda^{\alpha-1}t)^{q}, \quad as \; t\rightarrow \infty.
\end{eqnarray*}
\end{lemma}

\begin{definition}
	Let $ N^H(a, b)=\{N^{H}_{t}(a,b), \; t\geq 0\}$  be a mfBm and let  $S^{\lambda,\alpha}=\{S^{\lambda,\alpha}_{t}, \; t\geq 0\}$ be a TSS with index $\alpha\in(0, 1)$ and tempering parameter $\lambda>0$. The time-changed process of $N^H(a, b)$ by means of  $S^{\lambda,\alpha}$ is the process  $Y_{S^{\lambda,\alpha}}^{H}(a, b)=\{Y^{H}_{S^{\lambda,\alpha}_t}, \; t\geq 0\}$ defined by:
	\begin{eqnarray}
		 Y^{H}_{S^{\lambda,\alpha}_t}=N^{H}_{S^{\lambda,\alpha}_{t}}(a,b)=aB_{S^{\lambda,\alpha}_{t}}+bB^H_{S^{\lambda,\alpha}_{t}}, \quad (a,b) \in \mathbb{R}\times \mathbb{R}\backslash\{0\},
	\end{eqnarray}
	where the subordinator $S^{\lambda,\alpha}_{t}$ is assumed to be independent of both the Bm and the fBm.
\end{definition}
	\begin{notation}
Let $X$ and $Y$ be two random variables defined on the same probability space $(\Omega, \mathcal{F}, P).$ We denote the correlation coefficient $Corr(X, Y)$  by
	\begin{eqnarray}\label{qqq1}
Corr(X, Y)=\frac{Cov(X, Y)}{\sqrt{Var(X) Var(Y)}}.
\end{eqnarray}	
	\end{notation}
	
\begin{proposition}\label{pr11}
	Let   $Y_{S^{\lambda,\alpha}}^{H}(a, b)$ be the mfBm time-changed  by  $S^{\lambda,\alpha}$. Then by  Taylor's expansion we get, for fixed $s$ and large $t$,
		\begin{eqnarray}
		Cov(	Y^{H}_{S^{\lambda,\alpha}_t},	Y^{H}_{S^{\lambda,\alpha}_s})
		&\sim&
		\frac{1}{2}a^2s(\alpha\lambda^{\alpha-1})+ b^2Hs(\alpha\lambda^{\alpha-1})^{2H}t^{2H-1}, \ \ as \ \ t \rightarrow \infty.\label{eq11}	
		\end{eqnarray}
	\end{proposition}
	\begin{proof}
 For fixed $s$ and using  (\cite{KGW}, pp 195), the process $	Y^{H}_{S^{\lambda,\alpha}_t}$ follows
		\begin{eqnarray*}
			Cov(	Y^{H}_{S^{\lambda,\alpha}_t},	Y^{H}_{S^{\lambda,\alpha}_s})
			&=&
			\frac{1}{2}E\left[(	Y^{H}_{S^{\lambda,\alpha}_t})^2+(	Y^{H}_{S^{\lambda,\alpha}_s})^2-(	 Y^{H}_{S^{\lambda,\alpha}_t}-	 Y^{H}_{S^{\lambda,\alpha}_s})^2 \right]\nonumber\\ &=&\frac{1}{2}E\left[(N^{H}_{S^{\lambda,\alpha}_{t}}(a,b))^2+(N^{H}_{S^{\lambda,\alpha}_{s}}(a,b))^2-(N^{H}_{S^{\lambda,\alpha}_{t}}(a,b)-N^{H}_{S^{\lambda,\alpha}_{s}}(a,b))^2\right ] \nonumber\\
&=&\frac{1}{2}E\left[(aB_{S^{\lambda,\alpha}_{t}}+bB^H_{S^{\lambda,\alpha}_{t}}) ^2+( aB_{S^{\lambda,\alpha}_{s}}+bB^H_{S^{\lambda,\alpha}_{s}}) ^2\right ] \nonumber\\
&-&\frac{1}{2}E\left[\left(a(B_{S^{\lambda,\alpha}_{t}}-
B_{S^{\lambda,\alpha}_{s}})+b(B^H_{S^{\lambda,\alpha}_{t}}-B^H_{S^{\lambda,\alpha}_{s}})\right)^2\right ]. \nonumber\\
\end{eqnarray*}
Since $B^H$ has stationary increments, then
\begin{eqnarray}\label{qqq11}
	Cov(	Y^{H}_{S^{\lambda,\alpha}_t},	Y^{H}_{S^{\lambda,\alpha}_s})	&=&\frac{1}{2}E \left[ (aB_{S^{\lambda,\alpha}_{t}}+bB^H_{S^{\lambda,\alpha}_{t}}) ^2
			+( aB_{S^{\lambda,\alpha}_{s}}+bB^H_{S^{\lambda,\alpha}_{s}}) ^2-( aB_{S^{\lambda,\alpha}_{t-s}}+bB^H_{S^{\lambda,\alpha}_{t-s}})^2\right]
			 \nonumber\\
			&=&\frac{1}{2}E\left[(aB_{S^{\lambda,\alpha}_{t}})^2+( bB^H_{S^{\lambda,\alpha}_{t}})^2+2abB_{S^{\lambda,\alpha}_{t}} B^H_{S^{\lambda,\alpha}_{t}} \right]\nonumber\\
			&+&\frac{1}{2}E\left[(aB_{S^{\lambda,\alpha}_{s}})^2+( bB^H_{S^{\lambda,\alpha}_{s}})^2+2abB_{S^{\lambda,\alpha}_{s}} B^H_{S^{\lambda,\alpha}_{s}} \right]\nonumber\\
			&-&\frac{1}{2}E\left[(aB_{S^{\lambda,\alpha}_{t-s}})^2+( bB^H_{S^{\lambda,\alpha}_{t-s}})^2
			+2abB_{S^{\lambda,\alpha}_{t-s}} B^H_{S^{\lambda,\alpha}_{t-s}} \right].
			\end{eqnarray}
By the  independence of $B_{t}$ and $B_{t}^{H}$, we get
	\begin{eqnarray*}
		Cov(	Y^{H}_{S^{\lambda,\alpha}_t},	Y^{H}_{S^{\lambda,\alpha}_s})	 &=&\frac{a^2}{2}\left[E(B_{S^{\lambda,\alpha}_{t}})^2+E(B_{S^{\lambda,\alpha}_{s}})^2-E(B_{S^{\lambda,\alpha}_{t-s}})^2 \right]\nonumber\\
			 &+&\frac{b^2}{2}\left[E(B^H_{S^{\lambda,\alpha}_{t}})^2+E(B^H_{S^{\lambda,\alpha}_{s}})^2-E(B^H_{S^{\lambda,\alpha}_{t-s}})^2 \right]\nonumber\\
			&=&\frac{a^2}{2}E(B^\frac{1}{2}(1))^2\left[ E( S^{\lambda,\alpha}_{t}) +E( S^{\lambda,\alpha}_{s})-E( S^{\lambda,\alpha}_{t-s}) \right]\nonumber\\ &+&\frac{b^2}{2}E(B^H(1))^2\left[ E(S^{\lambda,\alpha}_{t}) ^{2H}+E( S^{\lambda,\alpha}_{s})^{2H}-E(S^{\lambda,\alpha}_{t-s})^{2H}\right].	
			\end{eqnarray*}
Hence for large $t$  and using Lemma \ref{ll1}, we have
	\begin{eqnarray*}
		Cov(	Y^{H}_{S^{\lambda,\alpha}_t},	Y^{H}_{S^{\lambda,\alpha}_s})	&\sim& \frac{a^2}{2}\left[ (\alpha\lambda^{\alpha-1}) t+E( S^{\lambda,\alpha}_{s})-(\alpha\lambda^{\alpha-1}) {(t-s)}\right]\nonumber \\
			&+&\frac{b^2}{2}\left[ (\alpha\lambda^{\alpha-1})^{2H} t^{2H}+E( S^{\lambda,\alpha}_{s})^{2H}-(\alpha\lambda^{\alpha-1})^{2H} {(t-s)}^{2H}\right]\nonumber\\		
			&=& \frac{a^2}{2} (\alpha\lambda^{\alpha-1}) t\left(\frac{s}{t}+E(S^{\lambda,\alpha}_{s})t^{-1}+O(t^{-2}) \right)\nonumber\\
			&+&\frac{b^2}{2}(\alpha\lambda^{\alpha-1})^{2H} t^{2H}\left(2H\frac{s}{t}+E( S^{\lambda,\alpha}_{s})^{2H}t^{-2H}+O(t^{-2}) \right)\nonumber \\
			&\sim&  \frac{a^2}{2} (\alpha\lambda^{\alpha-1}) s+ b^2Hs(\alpha\lambda^{\alpha-1})^{2H}t^{2H-1}.
		\end{eqnarray*}
	\end{proof}


\begin{proposition}\label{p123}
Let $ N^H(a, b)=\{N^{H}_{t}(a,b), \; t\geq 0\}$  be the mfBm of parameters $a,  b $ and $  H$. Let   $S^{\lambda,\alpha}=\{S^{\lambda,\alpha}_{t},\; t\geq 0\}$ be the TSS with index $\alpha\in(0, 1)$ and tempering parameter $\lambda>0$
and let   $Y_{S^{\lambda,\alpha}}^{H}(a, b)$ be the mfBm time-changed process  by means of $S^{\lambda,\alpha}.$ Then for fixed $s>0$ and $t \rightarrow \infty$, we get
 \begin{eqnarray*}
		E[(	Y^{H}_{S^{\lambda,\alpha}_t}-	Y^{H}_{S^{\lambda,\alpha}_s})^{2}]&\sim&	 \frac{1}{2}a^2t(\alpha\lambda^{\alpha-1})+ b^2H(\alpha\lambda^{\alpha-1})^{2H}t^{2H}-
		a^2s(\alpha\lambda^{\alpha-1})- 2b^2Hs(\alpha\lambda^{\alpha-1})^{2H}t^{2H-1}\\&+&\frac{1}{2}a^2s(\alpha\lambda^{\alpha-1})+ b^2H(\alpha\lambda^{\alpha-1})^{2H}s^{2H}.
	\end{eqnarray*}
	\end{proposition}
	\begin{proof} Let $s>0$ be fixed and $t \rightarrow \infty.$ Then by using Eq. \eqref{eq11}, we have
 \begin{eqnarray*}
		E[(	Y^{H}_{S^{\lambda,\alpha}_t}-	Y^{H}_{S^{\lambda,\alpha}_s})^{2}]
		&=&	 E\left[(Y^{H}_{S^{\lambda,\alpha}_t}-	Y^{H}_{S^{\lambda,\alpha}_s})(	Y^{H}_{S^{\lambda,\alpha}_t}-	 Y^{H}_{S^{\lambda,\alpha}_s}) \right]\\
		 &=&E\left[(	Y^{H}_{S^{\lambda,\alpha}_t})^2-	Y^{H}_{S^{\lambda,\alpha}_t}	Y^{H}_{S^{\lambda,\alpha}_s}-	 Y^{H}_{S^{\lambda,\alpha}_t}	Y^{H}_{S^{\lambda,\alpha}_s}+(	Y^{H}_{S^{\lambda,\alpha}_s})^2 \right]\\
		&=&E\left[(	Y^{H}_{S^{\lambda,\alpha}_t})^2-2	Y^{H}_{S^{\lambda,\alpha}_t}	Y^{H}_{S^{\lambda,\alpha}_s}+(	 Y^{H}_{S^{\lambda,\alpha}_s})^2 \right]\\
		&\sim&
		\frac{1}{2}a^2t(\alpha\lambda^{\alpha-1})+ b^2H(\alpha\lambda^{\alpha-1})^{2H}t^{2H}-
		a^2s(\alpha\lambda^{\alpha-1})\\&-& 2b^2Hs(\alpha\lambda^{\alpha-1})^{2H}t^{2H-1}+\frac{1}{2}a^2s(\alpha\lambda^{\alpha-1})+ b^2H(\alpha\lambda^{\alpha-1})^{2H}s^{2H}.
	\end{eqnarray*}
	\end{proof}

Now we discuss the long range dependence behavior of    $Y_{S^{\lambda,\alpha}}^{H}(a, b).$
\begin{definition}
	Note that, a finite variance stationary process $\{X_t,\;t\geq 0\}$ is said to have long range dependence property (Cont and Tankov \cite{Cont}), if $\sum_{k=0}^{\infty}\gamma_k=\infty$, where
	\begin{eqnarray*}
		\gamma_k=Cov(X_k,X_{k+1}).
	\end{eqnarray*}
\end{definition}

	In the following definition we give the equivalent definition for a non-stationary process $\{X_t,\;t\geq 0\}$.
	
\begin{definition}\label{d1}
	Let $s>0$ be fixed and $t>s$. Then process $\{X_t,\;t\geq 0\}$ is said to have long range dependence property property if
	\begin{eqnarray*}
		Corr(X_t, X_s)\sim c(s)t^{-d}, \ \ as \ \ t\rightarrow\infty,
	\end{eqnarray*}	
	where $c(s)$ is a constant depending on $s$ and $d\in(0,1)$.
\end{definition}

\begin{theorem} Let $ N^H(a, b)=\{N^{H}_{t}(a,b), \; t\geq 0\}$  be the mfBm of parameters $a,  b $ and $  H$. Let   $S^{\lambda,\alpha}=\{S^{\lambda,\alpha}_{t},\; t\geq 0\}$ be the TSS with index $\alpha\in(0, 1)$ and tempering parameter $\lambda>0$. Then  the time-changed mixed fractional Brownian motion by means of   $S^{\lambda,\alpha}$  has long range dependence  property for every $H>\frac{1}{2}$.
	\end{theorem}
	\begin{proof}
The process $Y_{S^{\lambda,\alpha}}^{H}(a, b)$ is not stationary, hence the Definition \ref{d1}  will be used to establish the long range dependence property.\\
	Let $\frac{1}{2}<H<1$.  Using Eqs. \eqref{qqq1}, \eqref{eq11} and by  Taylor's expansion we get, as $t\rightarrow\infty$
	\begin{eqnarray*}	
Corr(Y^{H}_{S^{\lambda,\alpha}_t},	Y^{H}_{S^{\lambda,\alpha}_s})
&\sim& \frac{	\frac{1}{2}a^2s(\alpha\lambda^{\alpha-1})+ b^2Hs(\alpha\lambda^{\alpha-1})^{2H}t^{2H-1}}
		{\sqrt{\left(\frac{1}{2}a^2(\alpha\lambda^{\alpha-1})t+ b^2H(\alpha\lambda^{\alpha-1})^{2H}t^{2H}\right)}\sqrt{	 E(	 Y^{H}_{S^{\lambda,\alpha}_s})^{2}}}\\&=&
\frac{\frac{1}{2}a^2s(\alpha\lambda^{\alpha-1})+b^2 Hs(\alpha\lambda^{\alpha-1})^{2H}t^{2H-1}}{\sqrt{b^2H(\alpha\lambda^{\alpha-1})^{2H}t^{2H}\left[
	\frac{a^2}{2b^2H}(\alpha\lambda^{\alpha-1})^{1-2H}t^{1-2H}+ 1\right] }\sqrt{	E(	Y^{H}_{S^{\lambda,\alpha}_s})^{2}}}\\
&\sim&\frac{	\frac{a^2}{2}H^{-\frac{1}{2}}s(\alpha\lambda^{\alpha-1})^{1-H}}{|b|\sqrt{	 E(	 Y^{H}_{S^{\lambda,\alpha}_s})^{2}}}t^{-H}
		+\frac{	 |b|H^{\frac{1}{2}}s(\alpha\lambda^{\alpha-1})^{H}}{\sqrt{	E(	Y^{H}_{S^{\lambda,\alpha}_s})^{2}}}t^{H-1}.		
	\end{eqnarray*}
Then the correlation function of $	Y^{H}_{S^{\lambda,\alpha}_t}$ decays like a mixture of power law $t^{-H}+t^{-(1-H)}$ and the  time-changed process $Y^{H}_{S^{\lambda,\alpha}}(a, b)$ exhibits  long range dependence  property for all $H>\frac{1}{2}$.
	\end{proof}
\begin{figure}[h!]
    \centering
    \includegraphics[width=0.9\linewidth]{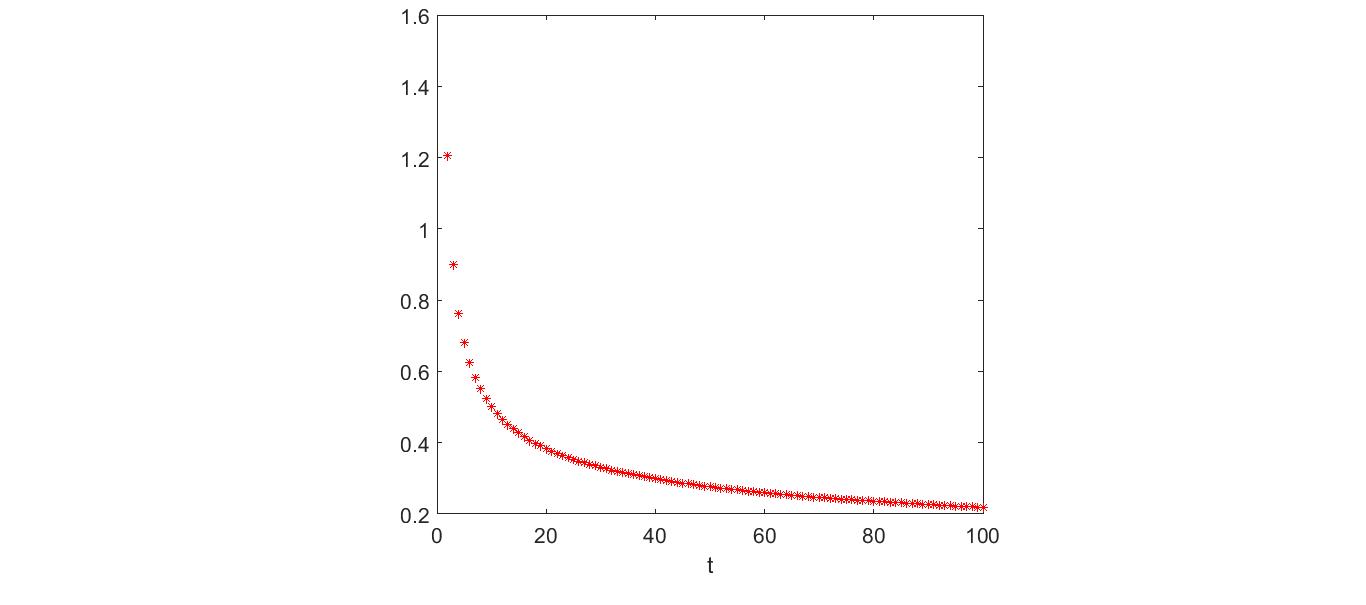}
    \caption{ The correlation function of mixed fractional Brownian motion time-changed by (TSS)  for $s=1$, $a=b=1$, $\lambda=0.1$, $\alpha=0.5$ and $H=0.7$.}
    \label{}
\end{figure}

\begin{remark}
	When $a=0$ and $b=1$ in Eqs. \eqref{eq11} and \eqref{qqq1}, we obtain
	\begin{eqnarray*}
				&&Cov(	Y^{H}_{S^{\lambda,\alpha}_t},	Y^{H}_{S^{\lambda,\alpha}_s})
				\sim Hs(\alpha\lambda^{\alpha-1})^{2H}t^{2H-1},\quad as \;  t\rightarrow\infty,
			\\&&	Corr(	Y^{H}_{S^{\lambda,\alpha}_t},	Y^{H}_{S^{\lambda,\alpha}_s})
				\sim Hs^{1-H}t^{H-1},\quad as \;  t\rightarrow\infty.
			\end{eqnarray*}
\end{remark}
Hence we obtain the result proved in \cite{KGW}
\begin{corollary}
The fractional Brownian motion time-changed by TSS  has long range dependence  property for every $H\in(0,1)$.
\end{corollary}

\section{MfBm time-changed by the gamma subordinator}
\begin{definition}
	Gamma process $\Gamma=\{\Gamma_t, t\geq\}$ is a Stationary independent increments process with gamma distribution. More precisely, the increment $\Gamma_{t+s}-\Gamma_s$ have density function
	\begin{eqnarray*}
		f(x,t)=\dfrac{1}{\Gamma(t/\nu)}x^{(t/\nu)-1}e^{-x}, \quad  x>0, \quad  \nu>0.
	\end{eqnarray*}
More detail about gamma subordinator can be founded in \cite{KWPS}.
\end{definition}
\begin{lemma} (see \cite{KWPS} for the proof)\label{ll2}\\
For $q>0,$ the asymptotic behavior of q-th order moments of   $\Gamma_{t}$ satisfies
\begin{eqnarray*}
E(\Gamma_{t})^{q}\sim \left(\frac{t}{\nu}\right)^{q}, \quad as \; t\rightarrow \infty.
\end{eqnarray*}
\end{lemma}
\begin{definition}
	Let $ N^H(a, b)=\{N^{H}_{t}(a,b), \; t\geq 0\}$  be a mfBm and let  $\Gamma=\{\Gamma_{t},\; t\geq 0\}$ be a gamma subordinator.  The time-changed process of $N^H(a, b)$ by means of $\Gamma$ is the process  $Y_{\Gamma}^{H}(a, b) =\{Y^{H}_{{\Gamma}_t}, \; t\geq 0\}$ defined by:
	\begin{eqnarray}
	Y^{H}_{{\Gamma}_t}=N^{H}_{\Gamma_{t}}(a,b)=aB_{\Gamma_{t}}+bB^H_{\Gamma_{t}}, \quad (a,b) \in \mathbb{R}\times \mathbb{R}\backslash\{0\},
	\end{eqnarray}
	where the subordinator $\Gamma_{t}$ is assumed to be independent of both the Bm and the fBm.
\end{definition}

\begin{proposition}\label{pr1}
Let $Y_{\Gamma}^{H}(a, b)$  be the mfBm time-changed by $\Gamma$.  Then we have
	\begin{enumerate}
		\item For $s<t$, the covariance function for the process $Y^{H}_{\Gamma_t}$ follows
	\begin{eqnarray*}
	Cov(Y^{H}_{{\Gamma}_t},Y^{H}_{{\Gamma}_s})&=&
\frac{a^2}{2}\left[\frac{\Gamma(1+t/\nu)}{\Gamma(t/\nu)}+\frac{\Gamma(1+s/\nu)}{\Gamma(s/\nu)}-\frac{\Gamma(1+(t-s)/\nu)}{\Gamma((t-s)/\nu)}\right]\nonumber\\
&+&\frac{b^2}{2}\left[\frac{\Gamma(2H+t/\nu)}{\Gamma(t/\nu)}+\frac{\Gamma(2H+s/\nu)}{\Gamma(s/\nu)}-\frac{\Gamma(2H+(t-s)/\nu)}{\Gamma((t-s)/\nu)}\right].\label{q1234}\nonumber
			\end{eqnarray*}
		\item For fixed $s$ and large $t$, the process $Y^{H}_{{\Gamma}_t}$ follows	
		\begin{eqnarray}
		Cov(Y^{H}_{{\Gamma}_t},Y^{H}_{{\Gamma}_s})\sim
		a^2\frac{s}{\nu}+ \frac{2b^2Hs}{\nu^{2H}}t^{2H-1}.	\label{qqq17}
		\end{eqnarray}
	\end{enumerate}
	\end{proposition}
	\begin{proof}
	\begin{enumerate}
		\item Let  $s<t$. Using similar procedure as the proof of Eq. \eqref{qqq11}, we get
		\begin{eqnarray*}
			Cov(Y^{H}_{{\Gamma}_t},Y^{H}_{{\Gamma}_s})&=&E(Y^{H}_{{\Gamma}_t}Y^{H}_{{\Gamma}_s})\nonumber\\
			&=&\frac{a^2}{2}\left[E(B_{\Gamma_{t}})^2+E(B_{\Gamma_{s}})^2-E(B_{\Gamma_{t-s}})^2 \right]
			+\frac{b^2}{2}\left[E(B^H_{\Gamma_{t}})^2+E(B^H_{\Gamma_{s}})^2-E(B^H_{\Gamma_{t-s}})^2 \right]\\
			 &=&\frac{a^2}{2}\left[\frac{\Gamma(1+t/\nu)}{\Gamma(t/\nu)}+\frac{\Gamma(1+s/\nu)}{\Gamma(s/\nu)}-\frac{\Gamma(1+(t-s)/\nu)}{\Gamma((t-s)/\nu)}\right]\nonumber\\
&+&\frac{b^2}{2}\left[\frac{\Gamma(2H+t/\nu)}{\Gamma(t/\nu)}+\frac{\Gamma(2H+s/\nu)}{\Gamma(s/\nu)}-\frac{\Gamma(2H+(t-s)/\nu)}{\Gamma((t-s)/\nu)}\right].
		\end{eqnarray*}
		
		\item Let $g(x)=\Gamma(x+2H)/\Gamma(x)$ and $f(x)=\Gamma(x+1)/\Gamma(x)$. By Taylor expansion and \cite{KWPS} we have
		\begin{eqnarray}
		\frac{g(x+h)}{g(x)}=1+2H(h/x)+H(2H-1)(h/x)^2+O(x^{-3}),\label{qqq123}
		\end{eqnarray}
		and
		\begin{eqnarray}
		\frac{f(x+h)}{f(x)}=1+(h/x)+O(x^{-2}).\label{qqqq1234}
		\end{eqnarray}
		For fixed $s$ and large $t$, using Eqs. \eqref{qqq17}, \eqref{qqq123} and \eqref{qqqq1234}, $Y^H_{\Gamma_t}$ follows
		\begin{eqnarray*}
			 Cov(Y^{H}_{{\Gamma}_t},Y^{H}_{{\Gamma}_s})&=&\frac{a^2}{2}f(t/\nu)\left[1+\frac{f(s/\nu)}{f(t/\nu)}-\frac{f((t-s)/\nu)}{f(t/\nu)}\right]
+\frac{b^2}{2}g(t/\nu)\left[1+\frac{g(s/\nu)}{g(t/\nu)}-\frac{g((t-s)/\nu)}{g(t/\nu)}\right]\\
&=&\frac{a^2}{2}(t/\nu)\left[1+\frac{f(s/\nu)}{f(t/\nu)}-\left(1-\frac{s}{t}+O(t^{-2})\right)\right]\\
&+&\frac{b^2}{2}(t/\nu)^{2H}\bigg[1+\frac{g(s/\nu)}{g(t/\nu)}-\bigg(1-2H\left(\frac{s}{t}\right)+H(2H-1)\left(\frac{s^2}{t^2}\right)+O(t^{-3})\bigg)\bigg] \\
		&\sim&
		a^2\frac{s}{\nu}+ \frac{2b^2Hs}{\nu^{2H}}t^{2H-1}.		
		\end{eqnarray*}

	\end{enumerate}
	\end{proof}

\begin{proposition}\label{p1}
Let $ N^H(a, b)=\{N^{H}_{t}(a,b), \; t\geq 0\}$  be the mfBm and let  $\Gamma=\{\Gamma_{t},\; t\geq 0\}$ be a gamma subordinator. Let  $Y_{\Gamma}^{H}(a, b) =\{Y^{H}_{\Gamma_t}, \; t\geq 0\}$ be the mfBm time-changed by means of    $\Gamma.$ Then for fixed $s>0$ and $t\rightarrow \infty$, we have
\begin{enumerate}
\item	
\begin{eqnarray*}
		E[(Y^{H}_{\Gamma_t}-Y^{H}_{\Gamma_s})^{2}]\sim	a^2\frac{t}{\nu}+ \frac{2b^2H}{\nu^{2H}}t^{2H}-	2a^2\frac{s}{\nu}- \frac{4b^2Hs}{\nu^{2H}}t^{2H-1}+	 a^2\frac{s}{\nu}+ \frac{2b^2H}{\nu^{2H}}s^{2H}.
	\end{eqnarray*}
  \item For $\frac{1}{2}<H<1$. The correlation function is given by \begin{eqnarray}		 Corr(Y^{H}_{\Gamma_t},Y^{H}_{\Gamma_s})\sim\frac{a^2(2H)^{-\frac{1}{2}}s}{\nu^{1-H}|b|\sqrt{E(Y^{H}_{\Gamma_s})^{2}}}t^{-H}
	+\frac{|b|(2H)^{\frac{1}{2}}s}{\nu^{H}\sqrt{E(Y^{H}_{\Gamma_s})^{2}}}t^{H-1}. \label{qqq10}
	\end{eqnarray}
  \end{enumerate}
	\end{proposition}
	\begin{proof}
	Let $s>0$ be fixed and large $t$. Then
\begin{enumerate}
\item Using Eq. \eqref{qqq17}, we have
	\begin{eqnarray*}
		E[(Y^{H}_{\Gamma_t}-Y^{H}_{\Gamma_s})^{2}]
		&=&E\left[(Y^{H}_{\Gamma_t})^2-2Y^{H}_{\Gamma_t}Y^{H}_{\Gamma_s}+(Y^{H}_{\Gamma_s})^2 \right]\\
		&\sim&
		a^2\frac{t}{\nu}+ \frac{2b^2H}{\nu^{2H}}t^{2H}-	2a^2\frac{s}{\nu}- \frac{4b^2Hs}{\nu^{2H}}t^{2H-1}+	 a^2\frac{s}{\nu}+ \frac{2b^2H}{\nu^{2H}}s^{2H}.
	\end{eqnarray*}
\item Let  $\frac{1}{2}<H<1$. Using Eqs. \eqref{qqq1}, \eqref{qqq17} and by  Taylor's expansion we get, as $t\rightarrow\infty$
	\begin{eqnarray*}
	Corr(Y^{H}_{\Gamma_t},Y^{H}_{\Gamma_s})&\sim&\frac{		a^2\frac{s}{\nu}+ \frac{2b^2Hs}{\nu^{2H}}t^{2H-1}}{\sqrt{(	 a^2\frac{t}{\nu}+ \frac{2b^2H}{\nu^{2H}}t^{2H})}\sqrt{
		E(Y^{H}_{\Gamma_s})^{2}} }\\&=&	 \frac{1}{|b|(2H)^\frac{1}{2}(\frac{t}{\nu})^H\sqrt{\left[1+\frac{a^2}{2b^2H\nu^{1-2H}}t^{1-2H} \right]}}
	\frac{[a^2\frac{s}{\nu}+ \frac{2b^2Hs}{\nu^{2H}}t^{2H-1}]}{\sqrt{
		E(Y^{H}_{\Gamma_s})^{2}}}\\
	&=&\frac{(2H)^{-\frac{1}{2}}(\frac{t}{\nu})^{-H}}{|b|\left[1+\frac{a^2}{2b^2H\nu^{1-2H}}t^{1-2H} \right]^\frac{1}{2}}
	\frac{[a^2\frac{s}{\nu}+ \frac{2b^2Hs}{\nu^{2H}}t^{2H-1}]}{\sqrt{
		E(Y^{H}_{\Gamma_s})^{2}}}\\
		&\sim&\frac{a^2(2H)^{-\frac{1}{2}}s}{\nu^{1-H}|b|\sqrt{
		E(Y^{H}_{\Gamma_s})^{2}}}t^{-H}
	+\frac{|b|(2H)^{\frac{1}{2}}s}{\nu^{H}\sqrt{
		E(Y^{H}_{\Gamma_s})^{2}}}t^{H-1}.
	\end{eqnarray*}
Hence the  correlation function of $Y^{H}_{\Gamma_t}$ decays like a mixture of power law $t^{-H}+t^{-(1-H)}.$
\end{enumerate}
	\end{proof}
	\begin{figure}[h!]
    \centering
    \includegraphics[width=0.9\linewidth]{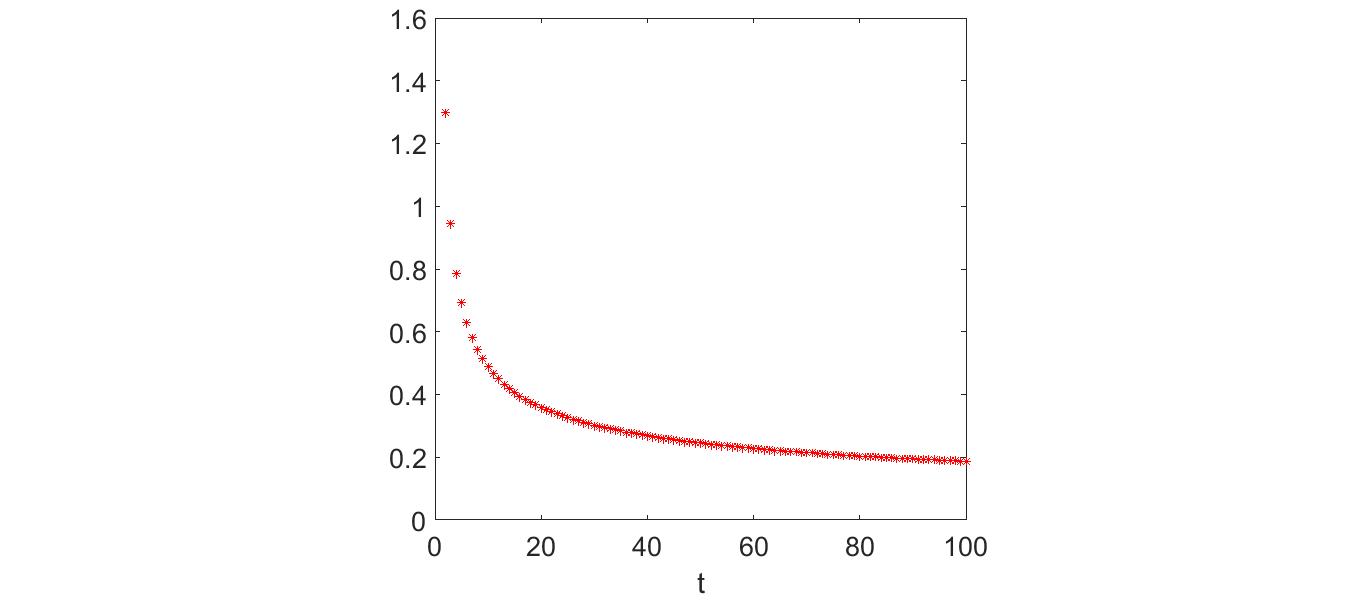}
    \caption{ The correlation function of mixed fractional Brownian motion time-changed by Gamma  for $s=1$, $a=b=1$, $v=0.75$ and $H=0.66$.}
    \label{}
\end{figure}
	
Using Definition  \ref{d1}  and Eq. \eqref{qqq10}  we obtain the  following result

\begin{theorem}
Let $ N^H(a, b)=\{N^{H}_{t}(a,b), \; t\geq 0\}$  be the mfBm of parameters $a,  b $ and $  H$. Let  $\Gamma=\{\Gamma_{t},\; t\geq 0\}$ be a gamma subordinator with parameter $\nu>0$.
Then the time-changed mixed fractional Brownian motion by means of   $\Gamma$ has long range dependence  property for every $H>\frac{1}{2}$.
  	\end{theorem}
\begin{remark} When $a=0$ and $b=1$ in Eqs. \eqref{qqq17} and \eqref{qqq1},  we get
		\begin{eqnarray*}
				&&Cov(Y^{H}_{\Gamma_t},Y^{H}_{\Gamma_s})
				\sim \frac{2Hs}{\nu^{2H}}t^{2H-1}, \ \ as \ \ t \rightarrow \infty,
				\\&& Corr(Y^{H}_{\Gamma_t},Y^{H}_{\Gamma_s})
				\sim\dfrac{2Hs}{\nu^H\sqrt{E(B_{s}^{\Gamma})^2}}t^{H-1}, \ \ as \ t\rightarrow\infty.
			\end{eqnarray*}
\end{remark}
Hence we obtain the result proved in \cite{KWPS}
\begin{corollary}
The fractional Brownian motion time-changed by gamma subordinator has long range dependence  property for every $H\in(0,1)$.
\end{corollary}


\end{document}